\documentclass[12pt]{amsart}
\usepackage[left=1.5cm, right=1.5cm, top=1.5cm, bottom=1.5cm]{geometry}

\usepackage[colorlinks=true,linkcolor=red]{hyperref}
\usepackage[square,numbers]{natbib}

\usepackage{amsmath,amsfonts,amssymb,amsthm,xfrac,subcaption,wasysym,enumitem,scalerel,stackengine}
\usepackage[utf8]{inputenc}
\usepackage[english]{babel}

\newtheorem{theorem}{Theorem}[section]
\newtheorem{lemma}[theorem]{Lemma}
\newtheorem{defn}[theorem]{Definition}

\newcommand{\R}{\mathbb{R}}
\newcommand{\C}{\mathbb{C}}

\stackMath
\newcommand\widecheck[1]{%
\savestack{\tmpbox}{\stretchto{%
  \scaleto{%
    \scalerel*[\widthof{\ensuremath{#1}}]{\kern-.6pt\bigwedge\kern-.6pt}%
    {\rule[-\textheight/2]{1ex}{\textheight}}
  }{\textheight}%
}{0.5ex}}%
\stackon[1pt]{#1}{\scalebox{-1}{\tmpbox}}%
}

\begin{document}
\linespread{1.5}
\title{Harnack Inequality for Magnetic Graphs}
\author{Sawyer Jack Robertson}
\date{\today}
\begin{abstract}
    For magnetic graphs satisfying connection curvature dimension inequality $CD^\sigma(n,\kappa)$, we prove a Harnack-type inequality for eigenfunctions of the graph magnetic Laplace operator in the manner of work done by Chung, Lin, Yau in 2014. Then we look at two applications; first a lower bound for the least eigenvalue in terms of curvature and extremal path/degree quantities, then to the magnetic Cheeger number of the graph.
\end{abstract}

\maketitle

\section{Introduction \& Preliminaries}

\subsection{Classical structures}
We consider a graph of the form $G=(V,E,p)$ where $V$ is the finite set of vertices, $E$ a set of undirected edges, and $p$ is an edge weight function. The edges are undirected, and we assume no loops (i.e., edges of the form $\{x,x\}$) or multiple edges. The weight of an edge $xy$ is denoted $p_{xy}$ and is assumed nonnegative and symmetric; the degree of a vertex $x$, denoted $d_x$, is the sum of all weights of edges incident to $x$. Adjacency between vertices $x,y$ is denoted $x\sim y$, and to avoid trivial complications, we require no isolated vertices (i.e., $d_x>0$ for each $x\in V$). \par
The \textit{oriented edge set} is the set of all pairs of adjacent vertices:
    \[E^{\text{or}}:=\{(x,y),(y,x):x,y\in V,\hspace{0.15cm}x\sim y\}.\]
Letting $V^{\C}$ and $V^\R$ denote the vector spaces of complex- and real-valued functions defined on $V$, respectively, the \textit{graph Laplace operator} $\Delta:V^\C\rightarrow V^\C$ is defined by
    \begin{equation}\label{laplacian-defn}
        (\Delta f)(x):=\frac{1}{d_x}\sum_{y\sim x}p_{xy}(f(y)-f(x)).
    \end{equation}
A nonzero function $f\in V^\C$ is said to be a \textit{harmonic eigenfunction} of $-\Delta$ with eigenvalue $\lambda$ provided $(-\Delta)f=\lambda f$. Note that $-\Delta$ will have nonnegative real eigenvalues since it is Hermitian and positive semidefinite. \par
To define Ricci curvature of a graph, we follow in the convention of \cite{LY-10} as adapted from work by Bakry, Emery\cite{BE-85}. This is done by defining a \textit{first curvature operator}, which is the billinear operator defined $\Gamma:V^\C\times V^\C\rightarrow V^\C$ by the formula
    \[2\Gamma(f,g):=\Delta(f\overline{g})-f\overline{\Delta g}-\Delta f\overline{g},\]
which may also be expressed pointwise by
    \begin{equation}\label{firstcurve-expr}
    2\left[\Gamma(f,g)\right](x)=\frac{1}{d_x}\sum_{y\sim x}p_{xy}(f(y)-f(x))\overline{(g(y)-g(x))}.
    \end{equation}
Note the conjugate symmetry of the operator. Setting for convenience $\Gamma(f):=\Gamma(f,f)$ we can define $|\nabla f|^2(x)$ via
    \begin{equation}\label{energy-defn}
        2\left[\Gamma(f)\right](x):=\frac{1}{d_x}\sum_{y\sim x}p_{xy}|f(y)-f(x)|^2:=|\nabla f|^2(x).
    \end{equation}
This will be called the \textit{energy} of $f$ at $x$. The \textit{Ricci curvature operator}, as in \cite{BE-85,CLY-14,LY-10}, is defined by iterating the first curvature operator; $\Gamma_2:V^\C\times V^\C\rightarrow V^\C$ is in turn defined by
    \begin{equation}\label{gamma2-defn}
        2\Gamma_2(f,g):=\Delta\Gamma(f,g)-\Gamma(f,\Delta g)-\Gamma(\Delta f, g).
    \end{equation}
A function $f\in V^\C$ is said to satisfy the curvature-dimension type inequality $CD(n,\kappa)$ for $n\in(1,\infty)$ and $\kappa\in\R\backslash\{0\}$ if
    \begin{equation}\label{CD-classical-ineq}
        \Gamma_2(f,f)\geq \frac{1}{n}|\Delta f|^2+\kappa\Gamma(f,f)
    \end{equation}
at each vertex $x\in V$. If this holds for every $f\in V^\C$ then we say that $G$ satisfies $CD(n,\kappa)$. If $CD(n,\kappa)$ holds for each $f$ in some class of functions $S\subset V^\C$, then we say $G$ satisfies the $S$-weak curvature dimension inequality $CD(n,\kappa,S)$.

\subsection{Magnetic structures}
In general, one can consider the analysis of functions $f:V\rightarrow\mathbb{F}^d$ where $\mathbb{F}=\R$ or $\C$ and $d\geq 1$ is some desired dimension; in turn, a \textit{graph connection} or \textit{signature} is a map $\sigma:E^{\text{or}}\rightarrow O_d$ where $O_d$ is the orthogonal group of appropriate dimension, satisfying $\sigma_{yx}=\sigma_{xy}^{-1}$. In this paper, we will restrict the scope to the case where $d=1$ and the signature $\sigma$ takes values in a cyclic group $S^1_\ell:=\{z\in\mathbb{C}:z^\ell=1\}$ for some $\ell\geq 1$; pairs $(G,\sigma)$ of this form are often called \textit{magnetic graphs}. A signature $\sigma$ taking values in $S^1_\ell$ will be called \textit{entire} if its range generates all of $S^1_\ell$; equivalently, if its range does not lie within a proper subgroup of $S^1_\ell$. The \textit{magnetic girth} of a magnetic graph, denoted $g^\sigma$, is defined to be the smallest directed cycle with the property that the product of the signature values along the edges of the cycle generates the group $S^1_\ell$. If the signature of a magnetic graph is not entire, or if no such cycle exists, $g^\sigma:=\infty$. A magnetic graph $(G,\sigma)$ is called \textit{balanced} provided that the product of the values of the signature along any (directed) cycle is 1; otherwise, $G$ is \textit{unbalanced}.\par
We define the \textit{magnetic Laplace operator} by $\Delta^\sigma:V^\C\rightarrow V^\C$ via
    \begin{equation}\label{maglaplace-defn}
        (\Delta^\sigma)f(x)=\frac{1}{d_x}\sum_{y\sim x}p_{xy}(\sigma_{xy}f(y)-f(x)),
    \end{equation}
As before, $-\Delta^\sigma$ will have nonnegative real eigenvalues. We can use the magnetic Laplace operator to approach the notion of Ricci curvature in the same manner as before, but taking into account the signature structure, following \cite{LMP-19}. The \textit{first magnetic curvature operator} $\Gamma^\sigma$ is the billinear operator defined in the one-dimensional case by
    \begin{equation}\label{conn-curve-defn}
        2\Gamma^\sigma(f,g):=\Delta^\sigma(f\overline{g})-f\left(\overline{\Delta^\sigma g}\right)-\left(\Delta^\sigma f\right)\overline{g},
    \end{equation}
with the \textit{magnetic Ricci curvature operator} given by
    \begin{equation}\label{2-conn-curve-defn}
        2\Gamma^\sigma_2(f,g):=\Delta\Gamma^\sigma(f,g)-\Gamma^\sigma(f,\Delta^\sigma g)-\Gamma^\sigma(\Delta^\sigma f, g).
    \end{equation}
In higher dimensions, Liu, et al. \cite[Eq. 1.17]{LMP-19} give more general definitions of the magnetic curvature; this is done in a natural manner by taking vector, rather than scalar, products.\par
Setting for convenience $\Gamma^\sigma(f):=\Gamma^\sigma(f,f)$ we can define $|\nabla^\sigma f|^2(x)$ via
    \begin{equation}\label{mag-energy-defn}
        2\left[\Gamma^\sigma(f)\right](x):=\frac{1}{d_x}\sum_{y\sim x}p_{xy}|\sigma_{xy}f(y)-f(x)|^2:=|\nabla^\sigma f|^2(x).
    \end{equation}
$f\in V^\C$ is said to satisfy the magnetic curvature-dimension type inequality $CD^\sigma(n,\kappa)$ for\break 
$n\in(1,\infty)$ and $\kappa\in\R\backslash\{0\}$ if
    \begin{equation}\label{CD-sigma-ineq}
        \Gamma_2^\sigma(f,f)\geq \frac{1}{n}|\Delta^\sigma f|^2+\kappa\Gamma^\sigma(f,f).
    \end{equation}
$(G,\sigma)$ is said to satisfy $CD^\sigma(n,\kappa)$ if the preceding holds for each $f\in V^\C$. \par

If $(G,\sigma)$ is a magnetic graph, and $\sigma$ takes values in a finite cyclic group $S^1_\ell$, we can define an associated combinatorial graph $\widehat{G}=(\widehat{V},\widehat{E},p)$ called the \textit{lift} or \textit{covering} graph, whose vertices are given by $\widehat{V}=V\times S^1_\ell$, and whose edges are defined via the relation
    \[(x_1,\xi_1)\sim(x_2,\xi_2)\iff x_1\sim x_2\text{ and }\xi_2=\xi_1\sigma_{x_1x_2},\]
with $p_{(x_1,\xi_1),(x_2,\xi_2)}:=p_{x_1x_2}$. We also define the vector space of functions $\widehat{V}^\C$ consisting of those defined on the vertices of the lift and taking complex values. When the context has fixed a magnetic graph $(G,\sigma)$, the Laplacian of its lift is denoted $\widehat{\Delta}$. \par
Before proceeding it will be advantageous to collect a few straightforward facts about how properties and structures of a magnetic graph relate to those on the lift.

\begin{lemma}\label{liftdiameter}
    Suppose $(G,\sigma)$ is an connected unbalanced magnetic graph with an entire signature $\sigma$ taking values in $S^1_\ell$. Suppose $G$ has diameter $D$, magnetic girth $g^\sigma<\infty$, and $\widehat{G}$ has diameter $\widehat{D}$. Then it holds
        \[\widehat{D}\leq 2D +\ell g^\sigma.\]
\end{lemma}

\begin{proof}
    As a preliminary note, suppose one has an oriented path in $G$ of length $n$ expressed as an ordered list of vertices $P=(x_0,x_1,\dots, x_{n})$. By the lift of $P$ initiating at level $\xi_0\in S^1_\ell$, which we denote $\widehat{P}_{\xi_0}$, we mean the path in $\widehat{G}$ given by the ordered list of vertices $$\widehat{P}_{\xi_0}:=\left((x_0,\xi_0),(x_1,\xi_0\sigma_{x_0x_1}),\dots (x_n,\xi_0\prod_{i=0}^{n-1}\sigma_{x_ix_{i+1}})\right).$$
    Now suppose one has two distinct vertices in $\widehat{G}$, say, $(y_1,\xi_1),(y_2,\xi_2)\in \widehat{V}$. Let $C$ be a directed cycle in $G$ realizing the magnetic girth of $G$, with signature product $\omega$ and length $g^\sigma$, containing the vertex $y^\ast$. Construct a directed path $P_1$ in $G$ connecting $y_1$ to $y^\ast$, and another directed path $P_2$ connecting $y^\ast$ to $y_2$. Let $\omega_1,\omega_2$ denote the products of the signature values along the paths $P_1,P_2$ resp.. Since $\sigma$ is entire, find an integer $m\geq 0$ for which $\omega^m=\xi_2\xi_1^{-1}\omega_1^{-1}\omega_2^{-1}$. Now form a path $P^\ast$ by concatenating in order $P_1$, then $m$ copies of $C$, followed by $P_2$. The product of the signature values along this path is $\xi_2\xi_1^{-1}$ by design, and it has length at most $2D+g^\sigma \ell$. $\widehat{P}^\ast_{\xi_1}$ connects $(y_1,\xi_1)$ and $(y_2,\xi_2)$ as desired.
\end{proof}

Suppose $G$ is a cycle on $2n$ vertices, with signature equal to $1$ everywhere except a single edge where it is equal to a primitive root in $S^1_\ell$. Then in this example $g^\sigma=2n$, $D=n$, so Lemma \ref{liftdiameter} supplies the estimate $\widehat{D}\leq 2n+n\ell$. One computes directly that $\widehat{D}=(2n)(\ell/2)=n\ell$, so the preceding estimate is sharp in its highest order terms.

\begin{lemma}
    Suppose $(G,\sigma)$ is an unbalanced magnetic graph, with an entire signature $\sigma$ taking values in a finite cyclic group $S^1_\ell$. Then $\widehat{G}$ is connected.
\end{lemma}

The verification of this lemma is straightforward, it being worthwhile to note that the condition on the range of the signature is more cosmetic than substantive; though a formality, dropping it can lead to disconnected lifts associated with unbalanced magnetic graphs.\par 
A useful tool is the \textit{lift embedding transformation}\hspace{0.25cm} $\widehat{  }:V^\C\rightarrow\widehat{V}^\C$ defined by $\widehat{f}(x,\xi):=\xi f(x)$. The image of the lift embedding transformation will be denoted $W\subset \widehat{V}^\C$.
    \begin{lemma}\label{lift-vs-magnetic-quantities}
        Suppose $(G,\sigma)$ is a magnetic graph. For each $f\in V^\C$ it holds
            \begin{enumerate}[label=(\roman*)]
                \item $|\nabla \widehat{f}|^2(x,\xi)=|\nabla^\sigma f|^2(x)$
                \item $(\widehat{\Delta}\widehat{f})(x,\xi)=\xi(\Delta^\sigma f)(x)$.
            \end{enumerate}
    \end{lemma}
In \cite[3.7]{LMP-19}, the authors obtained the following relationship between the $CD(n,\kappa)$ inequality for a covering graph and the $CD^\sigma(n,\kappa)$ inequality for the original connection graph. We will rephrase here, adapted to this terminology, for completeness:
    \begin{lemma}\label{weakercurvdimrel}
        Let $(G,\sigma)$ be a magnetic graph. If $\widehat{G}$ satisfies $CD(n,\kappa)$ then $G$ satisfies $CD^\sigma(n,\kappa)$.
    \end{lemma}
We have the following two lemmas giving a partial converse to Lemma \ref{weakercurvdimrel}.
    \begin{lemma}\label{localliftofcd}
        Let $(G,\sigma)$ be a magnetic graph. If $f\in V^\C$ satisfies $CD^\sigma(n,\kappa)$ then $\widehat{f}$ satisfies $CD(n,\kappa)$.
    \end{lemma}
    \begin{lemma}\label{weakliftofcd}
        Let $(G,\sigma)$ be a magnetic graph. If $G$ satisfies $CD^\sigma(n,\kappa)$, then $\widehat{G}$ satisfies the $W$-weak curvature dimension inequality $CD(n,\kappa,W)$.
    \end{lemma}
The proof of Lemma \ref{localliftofcd} is routine, relying on Lemma \ref{lift-vs-magnetic-quantities}, and Lemma \ref{weakliftofcd} follows immediately thereafter. 
    \begin{lemma}\label{lifteigenfunction}
        Suppose $(G,\sigma)$ is a magnetic graph. 
        If $f$ is an eigenfunction for $\Delta^\sigma$ with eigenvalue $\lambda$ then $\widehat{f}$ is an eigenfunction for $\widehat{\Delta}$ with eigenvalue $\lambda$. 
    \end{lemma}
Again, the proof of Lemma \ref{lifteigenfunction} is routine, relying on Lemma \ref{lift-vs-magnetic-quantities}.


\subsection{Summary of results}
This paper is an adaptation of the Harnack inequality of Chung, Lin, Yau\cite[Thm. 3.3]{CLY-14} to eigenfunctions of the magnetic Laplace operator for simple, connected, unbalanced magnetic graphs. They proved that for simple connected graphs satisfying the curvature dimension inequality $CD(n,\kappa)$, the following holds at each $x\in V$:
    \begin{equation}\label{clasicalharnack}
        \frac{1}{d_x}\sum_{y\sim x}|f(y)-f(x)|^2\leq \left(\left(8-\frac{2}{n}\right)\lambda-4\kappa\right)\max_{z\in V}|f|^2(z).
    \end{equation}
where $f$ is any eigenfunction of $-\Delta$ with nontrivial eigenvalue $\lambda>0$. We supply in section \ref{harnackproof} a proof of the same inequality extended to eigenfunctions of  $f:V\rightarrow\C$ of $-\Delta$, and apply this result to the lift associated to a simple, connected, unbalanced magnetic graph satisfying $CD^\sigma(n,\kappa)$ to obtain, at each $x\in V$,
    \[
        \frac{1}{d_x}\sum_{y\sim x}|f(x)-\sigma_{xy}f(y)|^2\leq \left(\left(8-\frac{2}{n}\right)\lambda-4\kappa\right)\max_{z\in V}|f|^2(z).
    \]
where $f:V\rightarrow\C$ is any eigenfunction of $-\Delta^\sigma$ with nontrivial eigenvalue $\lambda>0$. In section \ref{applications}, we discuss two applications. First, invoking an argument in \cite[Thm. 3.5]{CLY-14}, we derive the eigenvalue bound
    \[\lambda\geq \frac{1+4\kappa d \left(2D +\ell g^\sigma\right)^2}{d(8-\frac{2}{n})\left(2 +\ell g^\sigma\right)^2}\]
for the least eigenvalue of $-\Delta^\sigma$, for a simple magnetic graph with signature in the cyclic group $S^1_\ell$ of order $\ell$ satisfying $CD^\sigma(n,\kappa)$, with maximum degree $d$, and diameter $D$. The second application is a lower bound on the first magnetic Cheeger number $h^\sigma_1$, \textit{c.f.} definition \eqref{cheegerno}, of the graph:
    \[\frac{1+4\kappa d \left(2D +\ell g^\sigma\right)^2}{d(16-\frac{4}{n})\left(2D +\ell g^\sigma\right)^2}\leq h_1^\sigma.\]
One can heuristically think of the first Cheeger number as quantifying the extent to which the graph is balanced. This can be made precise, see e.g.\cite[Thm. 6.4]{LLPP-15}.

\par
Worth noting is that the Harnack inequality in equation \eqref{clasicalharnack} was strengthened by Chung and Yau\cite{CY-17} in 2017. It remains open whether this new result can be formulated for the connection graph case.

\section{Harnack inequality}\label{harnackproof}

The following arguments are made in the same manner as Chung, Lin, Yau\cite[Lemma 3.1, Thm. 3.2]{CLY-14}, with small adjustments made throughout to account for the complexity of the function values. We provide the computations for completeness.

\begin{lemma}\label{lemma-energy-laplace}
Let $G$ be a finite connected graph and suppose $f\in V^\C$ satisfies $CD(n,\kappa)$. Then at each vertex $x\in V$, 
    \[\left(\frac{4}{n}-2\right)|\Delta f(x)|^2+(2+2\kappa)|\nabla f|^2(x)\leq \frac{1}{d_x}\sum_{y\sim x}\frac{p_{xy}}{d_y}\sum_{z\sim y}p_{yz}|f(x)-2f(y)+f(z)|^2.\]
\end{lemma}

\begin{proof}
First we compute the Laplacian of $\Gamma(f)$:
    \[\begin{split}
        2\Delta\left[\Gamma(f)\right](x)&=\frac{1}{d_x}\sum_{y\sim x}p_{xy}\left[\Gamma(f)(y)-\Gamma(f)(x)\right]=\frac{1}{d_x}\sum_{y\sim x}p_{xy}\left[|\nabla f|^2(y)-|\nabla f|^2(x)\right]
        \\&=\frac{1}{d_x}\sum_{y\sim x}\frac{p_{xy}}{d_y}\sum_{z\sim y}p_{yz}\left[|f(z)-f(y)|^2-|f(y)-f(x)|^2\right]
    \end{split}\]
Using the straightforward expansion
    \[\begin{split}
    |f(z)-f(y)|^2-|f(y)-f(x)|^2&=|f(x)-2f(y)+f(z)|^2-(f(x)-f(y))\overline{(f(x)-2f(y)+f(z))}\\
    &\hspace{1.5cm}-\overline{(f(x)-f(y))}(f(x)-2f(y)+f(z)),
    \end{split}\]
we obtain
    \[\begin{split}
        2\Delta\left[\Gamma(f)\right](x)&=\frac{1}{d_x}\sum_{y\sim x}\frac{p_{xy}}{d_y}\sum_{z\sim y}p_{yz}|f(x)-2f(y)+f(z)|^2\\
        &+\frac{2}{d_x}\sum_{y\sim x}\frac{p_{xy}}{d_y}\sum_{z\sim y}p_{yz}\mathfrak{Re}\left\{{\left(f(y)-f(x)\right)\overline{\left(f(x)-2f(y)+f(z)\right)}}\right\}.
    \end{split}\]
Recalling equation \eqref{firstcurve-expr}, we have
    \[\begin{split}
        2\left[\Gamma(f,\Delta f)\right](x)&=\frac{1}{d_x}\sum_{y\sim x}p_{xy}\left(f(y)-f(x)\right)\overline{\left(\Delta f(y)-\Delta f(x)\right)}\\
        &=-|\Delta f(x)|^2+\frac{1}{d_x}\sum_{y\sim x}\frac{p_{xy}}{d_y}\sum_{z\sim y}{p_{yz}}(f(y)-f(x))\overline{\left(f(z)-f(y)+(f(y)-f(x))-(f(y)-f(x))\right)}\\
        &=-|\Delta f(x)|^2+|\nabla f|^2(x)+\frac{1}{d_x}\sum_{y\sim x}\frac{p_{xy}}{d_y}\sum_{z\sim y}p_{yz}(f(y)-f(x))\overline{(f(x)-2f(y)+f(z))}\\
    \end{split}\]
From equation \eqref{gamma2-defn}, we have
    \[\begin{split}
        \Gamma_2(f)(x)&=\frac{1}{2}\left[\Delta\Gamma(f)(x)-2\mathfrak{Re}\left\{ \Gamma(f,\Delta f)(x)\right\}\right]\\
        &=\frac{1}{2}\left[|\Delta f(x)|^2-|\nabla f|^2(x)+\frac{1}{2d_x}\sum_{y\sim x}\frac{p_{xy}}{d_y}\sum_{z\sim y}p_{yz}|f(x)-2f(y)+f(z)|^2\right]
    \end{split}\]
Since $G$ satisfies $CD(n,\kappa)$,
    \[\frac{1}{2}\left[|\Delta f(x)|^2 -|\nabla f|^2(x)+\frac{1}{2d_x}\sum_{y\sim x}\frac{p_{xy}}{d_y}\sum_{z\sim y}p_{yz}|f(x)-2f(y)+f(z)|^2\right]\geq \frac{1}{n}|\Delta f(x)|^2+\kappa\Gamma(f)(x).   \]
Isolating the Laplacian and $|\nabla f|^2(x)$ terms on the LHS, along with manipulating constants appropriately, provide the lemma.
\end{proof}

\begin{lemma}\label{lemma-laplace-of-square}
Let $G$ be a simple connected graph and let $f\in V^\C$. Then it holds at each $x\in V$,
    \[\left(\Delta |f|^2\right)(x)=|\nabla f|^2(x)-f(x)\overline{\Delta f(x)}-\overline{f(x)}\Delta f(x).\]
If $f$ is an eigenfunction for $-\Delta$ with eigenvalue $\lambda$, this becomes
    \[\left(-\Delta |f|^2\right)(x)=2\lambda|f|^2(x)-|\nabla f|^2(x).\]
\end{lemma}

\begin{proof}
We have
    \[\begin{split}
    |f(y)|^2-|f(x)|^2 &=f(y)\overline{f(y)}-f(x)\overline{f(x)}\\
    &=-\left[f(x)\overline{(f(x)-f(y))}+\overline{f(x)}(f(x)-f(y))\right]\\
    &+\left[f(x)\overline{f(x)}+f(y)\overline{f(y)}-f(y)\overline{f(x)}-f(x)\overline{f(y)}\right]\\
    &=|f(x)-f(y)|^2-f(x)\overline{(f(x)-f(y))}-\overline{f(x)}(f(x)-f(y))
    \end{split}\]
so
    \[\begin{split}
    \left(\Delta|f|^2\right)(x)&=\frac{1}{d_x}\sum_{y\sim x}p_{xy}\left(|f(y)|^2-|f(x)|^2\right)\\
    &=\frac{1}{d_x}\sum_{y\sim x}p_{xy} \left[|f(x)-f(y)|^2-f(x)\overline{(f(x)-f(y))}-\overline{f(x)}(f(x)-f(y))\right]\\
    &=|\nabla f(x)|^2-f(x)\overline{\left(-\Delta f(x)\right)}-\overline{f(x)}\left(-\Delta f(x)\right).
    \end{split}\]
\end{proof}

\begin{theorem}
Let $G$ be a finite connected graph and suppose $f\in V^\C$ is a harmonic eigenfunction of $-\Delta$ with nontrivial eigenvalue $\lambda>0$ satisfying $CD(n,\kappa)$. Then the following holds at each $x\in V$ and $\alpha> 2-2{\kappa}/{\lambda}$:
    \begin{equation}
        |\nabla f|^2(x)+\alpha\lambda |f|^2(x)\leq \frac{(\alpha^2-\frac{4}{n})\lambda+2\kappa\alpha}{(\alpha-2)\lambda+2\kappa}\lambda\max_{z\in V}|f|^2(z).
    \end{equation}
\end{theorem}

\begin{proof}
Using Lemma \ref{lemma-energy-laplace} and its proof,
    \[\begin{split}
    (-\Delta)|\nabla f|^2(x)&=(-2)\Delta\left[\Gamma(f)\right](x)\\
    &=-\frac{1}{d_x}\sum_{y\sim x}\frac{p_{xy}}{d_y}\sum_{z\sim y}p_{yz}|f(x)-2f(y)+f(z)|^2\\
    &-\frac{2}{d_x}\sum_{y\sim x}\frac{p_{xy}}{d_y}\sum_{z\sim y}p_{yz}\mathfrak{Re}\left\{{\left(f(y)-f(x)\right)\overline{\left(f(x)-2f(y)+f(z)\right)}}\right\}\\
    &\leq -\left(\frac{4}{n}-2\right)|\Delta f(x)|^2-(2+2\kappa)|\nabla f|^2(x)\\
    &-\frac{2}{d_x}\sum_{y\sim x}\frac{p_{xy}}{d_y}\sum_{z\sim y}p_{yz}\mathfrak{Re}\left\{{\left(f(y)-f(x)\right)\overline{\left(f(x)-2f(y)+f(z)\right)}}\right\}\\
    &=\lambda^2\left(2-\frac{4}{n}\right)|f|^2(x)-2\kappa|\nabla f|^2(x)-2\mathfrak{Re}\left\{\frac{1}{d_x}\sum_{y\sim x}p_{xy}(f(y)-f(x))(-\lambda\overline{f(y)})\right\}\\
    &=\lambda^2\left(2-\frac{4}{n}\right)|f|^2(x)-2\kappa|\nabla f|^2(x)\\
    &\hspace{.35cm}-2\mathfrak{Re}\left\{\frac{1}{d_x}\sum_{y\sim x}p_{xy}(f(y)-f(x))(-\lambda\overline{f(y)}+\lambda\overline{ f(x)}-\lambda\overline{ f(x)})\right\}\\
    &=\lambda^2\left(2-\frac{4}{n}\right)|f|^2(x)-2\kappa|\nabla f|^2(x)+2\lambda|\nabla f|^2(x)+2\lambda^2|f|^2(x)\\
    &=(2\lambda -2\kappa)|\nabla f|^2(x)-\frac{4}{n}\lambda^2|f|^2(x)
    \end{split}\]
Combining the preceding inequality and Lemma \ref{lemma-laplace-of-square}, we get, for $\alpha>0$,
    \[\begin{split}
        (-\Delta)(|\nabla f|^2(x)+\alpha\lambda|f|^2(x))&\leq (2\lambda -2\kappa)|\nabla f|^2(x)-\frac{4}{n}\lambda^2|f|^2(x)+2\alpha\lambda^2|f|^2(x)-\alpha\lambda|\nabla f|^2(x)\\
        &\leq\left(2\lambda-\alpha\lambda-2\kappa\right)|\nabla f|^2(x)+\left(2\alpha-\frac{4}{n}\right)\lambda^2|f|^2(x)
    \end{split}\]
    
Let $v\in V$ satisfy
    \[|\nabla f|^2(v)+\alpha\lambda|f|^2(v)=\max_{z\in V}\left[|\nabla f|^2(z)+\alpha\lambda|f|^2(z)\right].\]
Then since $v$ maximizes the expression over $V$, 
    \[\begin{split}
        0&\leq (-\Delta)\left(|\nabla f|^2(v)+\alpha\lambda|f|^2(v)\right)\\
        &\leq \left(2\lambda-\alpha\lambda-2\kappa\right)|\nabla f|^2(v)+\left(2\alpha-\frac{4}{n}\right)\lambda^2|f|^2(v).\\
    \end{split}\]
So, for $\alpha> 2-2\kappa/\lambda$,
    \[|\nabla f|^2(v)\leq\frac{2\alpha-\frac{4}{n}}{
    \lambda(\alpha-2)+2\kappa}\lambda^2|f|^2(v).\]
By our choice of $x\in V$,
    \[\begin{split}
    |\nabla f|^2(x)+\alpha\lambda|f|^2(x)&\leq |\nabla f|^2(v)+\alpha\lambda|f|^2(v)\\
    &\leq \frac{2\alpha-\frac{4}{n}}{\lambda(\alpha-2)+2\kappa}\lambda^2|f|^2(v)+\alpha\lambda|f|^2(v)\\
    &\leq\frac{(\alpha^2-\frac{4}{n})\lambda+2\kappa\alpha}{(\alpha-2)\lambda+2\kappa}\lambda\max_{z\in V}|f|^2(z).
    \end{split}\]
\end{proof}

Again following \cite[Thm. 3.3]{CLY-14}, choose $\alpha=4-2\kappa/\lambda$ to obtain the following:

\begin{theorem}\label{complexCLY}
Let $G$ be a finite connected graph and suppose $f\in V^\C$ is a harmonic eigenfunction of $-\Delta$ with nontrivial eigenvalue $\lambda>0$ satisfying $CD(n,\kappa)$. Then the following holds at each $x\in V$:
    \[
        |\nabla f|^2(x)\leq \left(\left(8-\frac{2}{n}\right)\lambda-4\kappa\right)\max_{z\in V}|f|^2(z).
    \]
\end{theorem}

Applying Lemmas \ref{lift-vs-magnetic-quantities}, \ref{weakliftofcd}, and Theorem \ref{complexCLY}, the following is immediately proved.

\begin{theorem}
Let $(G,\sigma)$ be a finite, connected, unbalanced magnetic graph with an entire signature $\sigma$ taking values in $S^1_\ell$. If $(G,\sigma)$ satisfies $CD^\sigma(n,\kappa)$ and $f\in V^\C$ is a harmonic eigenfunction of $-\Delta^\sigma$ with nontrivial eigenvalue $\lambda>0$, then the following holds at each $x\in V$:
    \[
        |\nabla^\sigma f|^2(x)\leq \left(\left(8-\frac{2}{n}\right)\lambda-4\kappa\right)\max_{z\in V}|f|^2(z).
    \]
\end{theorem}

\section{Eigenvalue estimate and application to Magnetic Cheeger number}\label{applications}
We will now apply the Harnack inequality in the manner of Chung, Lin, and Yau\cite{CLY-14} to obtain a lower bound for  the eigenvalues of the magnetic Laplacian. 

\begin{theorem}[Chung, Lin, Yau]\label{chungyauestimate}
    Suppose $G=(V,E,p)$ is a finite connected graph, and suppose $f\in V^\R$ is a harmonic eigenfunction for $-\Delta$ with eigenvalue $\lambda>0$, satisfying $CD(n,\kappa)$. Then it holds
        \[\lambda\geq\frac{1+4\kappa d D^2}{d(8-\frac{2}{n})D^2}.\]
\end{theorem}

This fact was proved by Chung, Yau\cite[Thm. 3.5]{CLY-14} for the case where $\Delta$ is an operator on $V^\R$. The authors identify a path connecting the vertices at which the extreme values of the eigenfunction in question are attained. The Harnack inequality is then used to estimate the energy along the path, and in turn, the eigenvalue. In the case where $\Delta$ is an operator on $V^\C$, the estimate still holds. In allowing the function values to be complex, an extremality argument is replaced with a convexity argument in identifying the vertices between which the modulus of the difference of the function values is maximal, but the approach is otherwise identical. 

\begin{theorem}\label{magnetic-eigenvalue-estimate}
    Suppose $(G,\sigma)$ is a connected, simple, unabalanced magnetic graph with entire signature $\sigma$ taking values in a cyclic group $S^1_\ell$, $\ell\geq 2$, satisfying $CD^\sigma(n,\kappa)$. Let $G$ have diameter $D$, maximum degree $d$, magnetic girth $g^\sigma<\infty$. Suppose $\lambda>0$ is a nontrivial eigenvalue for $-\Delta^\sigma$. Then it holds
        \[\lambda\geq\frac{1+4\kappa d \left(2D +\ell g^\sigma\right)^2}{d(8-\frac{2}{n})\left(2D +\ell g^\sigma\right)^2}.\]
\end{theorem}

\begin{proof}
    Suppose $f\in V^\C$ is a harmonic eigenfunction for $-\Delta^\sigma$ with eigenvalue $\lambda$. Then by Lemma \ref{lifteigenfunction}, $\widehat{f}\in\widehat{V}^\C$ is a harmonic eigenfunction for $\widehat{\Delta}$, the Laplacian for the lift, with the same eigenvalue. Moreover, it satisfies $CD(n,\kappa)$ by Lemma \ref{localliftofcd}. So, by Theorem \ref{chungyauestimate} and Lemma \ref{liftdiameter}, it holds
        \[\lambda\geq\frac{1+4\kappa d \widehat{D}^2}{d(8-\frac{2}{n})\widehat{D}^2}\geq\frac{1+4\kappa d \left(2D +\ell g^\sigma\right)^2}{d(8-\frac{2}{n})\left(2 +\ell g^\sigma\right)^2}. \]
\end{proof}

In this application, we use the preceding eigenvalue estimate to find a lower bound on the magnetic Cheeger number. The context is the work by Lange, Liu, Peyerimhoff, and Post\cite{LLPP-15}, whose estimate of the magnetic Cheeger number provides the link between the Cheeger number and the least eigenvalue of $\Delta^\sigma$. We note that in their approach, the vertex set of the graph is weighted and the signature takes values in an arbitrary group. Here we consider the case where the vertices are weighted according to their degrees, and the signature group is cyclic.

\begin{defn}
    Suppose $(G,\sigma)$ is a simple magnetic graph with signature in a cyclic group $S^1_\ell$,  and $V_1\subset V$ is nonempty with $(V_1,E_1)$ its induced subgraph. Then we define the frustration index of $V_1$, denoted $\iota^\sigma(V_1)$ to be
        \[\iota^\sigma(V_1):=\min_{\tau:V_1\rightarrow S^1_\ell}\sum_{\{x,y\}\in E_1}p_{xy}|\tau(x)-\sigma_{xy}\tau(y)|.\]
\end{defn}

The frustration index can be thought of as a measure of the balancedness of the subset $V_1$. This can be made explicit, as discussed in \cite{LLPP-15} in the context of the historical work by Harary who is widely considered to have been the first to formalize signed graphs (in the setting of signed social networks\cite{CH-56}). In particular, if $\sigma$ is taking values in $S^1_2$, i.e. the $\pm 1$ group, then
    \[\iota^\sigma(V)=2e^\sigma_{\text{min}}(V),\]
where $e^\sigma_{\text{min}}(V)$ is the minimum number of edges needed to be removed from $G$ to make it balanced.

\begin{defn}\label{cheegerno}
    Suppose $(G,\sigma)$ is a simple magnetic graph with signature in a cyclic group $S^1_\ell$. The magnetic Cheeger number $h_1^\sigma$ is defined by
        \[h_1^\sigma :=\min_{\varnothing\neq V_1\subset V}\frac{\iota^\sigma(V_1)+\sum_{\{x,y\}\in E(V_1,V_1^c)}p_{xy}}{\sum_{x\in V_1}d_x},\]
    where $E(V_1,V_1^c):=\left\{\{x,y\}\in E:x\in V_1,y\notin V_1\right\}$.
\end{defn}

Lange, Liu, Peyerimhoff, and Post\cite[Thm. 4.1]{LLPP-15} proved the following Cheeger inequality relating $h_1^\sigma$ to $\Delta^\sigma$.

\begin{theorem}[Cheeger's Inequality]
    Suppose $(G,\sigma)$ is a simple magnetic graph and $\sigma$ takes values in a finite cyclic group $S^1_\ell$. Let $\lambda$ be the least eigenvalue of $-\Delta^\sigma$, and $d$ the maximum degree in the graph. Then
        \[\frac{1}{2}\lambda\leq h_1^\sigma\leq2\sqrt{2d\lambda}.\]
\end{theorem}

Now in conjunction with the eigenvalue estimate in Theorem \ref{magnetic-eigenvalue-estimate}, and the preceding, we obtain

\begin{theorem}
    Suppose $(G,\sigma)$ is a connected, simple, unabalanced magnetic graph with entire signature $\sigma$ taking values in a cyclic group $S^1_\ell$, $\ell\geq 2$, satisfying $CD^\sigma(n,\kappa)$. Let $G$ have diameter $D$, maximum degree $d$, and magnetic girth $g^\sigma<\infty$. Then
        \[\frac{1+4\kappa d \left(2D +\ell g^\sigma\right)^2}{d(16-\frac{4}{n})\left(2D +\ell g^\sigma\right)^2}\leq h_1^\sigma.\]
\end{theorem}

\bibliography{references}
\bibliographystyle{plainnat}
    
\end{document}